\documentclass{amsart}

\usepackage[latin1]{inputenc}
\usepackage[T1]{fontenc}
\usepackage{amsmath,amsfonts,amssymb,amsthm}
\usepackage{graphicx}
\usepackage[all]{xy}

\theoremstyle{plain}
\newtheorem{theorem}{Theorem}[section]
\newtheorem*{theo}{Theorem A}
\newtheorem*{theoB}{Theorem B}
\newtheorem{conj}{Conjecture}[section]

\newtheorem{prop}[theorem]{Proposition}
\newtheorem{cor}[theorem]{Corollary}

\theoremstyle{definition}
\newtheorem{defin}[theorem]{Definition}

\theoremstyle{remark}
\newtheorem{rmk}[theorem]{Remark}
\newtheorem*{no}{Notation}
\newtheorem*{ack}{Acknowledgements}

\renewcommand{\hat}[1]{\widehat{#1}}
\renewcommand{\tilde}[1]{\widetilde{#1}}
\DeclareMathOperator{\Imm}{Im}
\renewcommand{\Im}{\Imm}

\newcommand{\set}[1]{{\left\{#1\right\}}}
\newcommand{\pa}[1]{{\left(#1\right)}}

\newcommand{\ra}{\rightarrow}

\newcommand{\sheaf}[2]{\mathcal{#1}_{#2}}
\newcommand{\fascio}[1]{\sheaf{O}{#1}}
\newcommand{\isom}{\xrightarrow{\sim}}

\newcommand{\pairing}{\langle \cdot, \cdot \rangle}
\newcommand{\m}[1]{\mathcal{#1}}

\DeclareMathOperator{\rk}{rk}
\DeclareMathOperator{\Hom}{Hom}
\DeclareMathOperator{\End}{End}
\DeclareMathOperator{\Aut}{Aut}

\DeclareMathOperator{\pic}{Pic}

\DeclareMathOperator{\Inv}{Inv}
\DeclareMathOperator{\Id}{Id}
\DeclareMathOperator{\NS}{NS}

\title[On the genus of curves in a Jacobian variety]{On the genus of curves in a Jacobian variety}
\author{Valeria Ornella Marcucci}
\address{Dipartimento di Matematica ``F. Casorati''\\
 	Universit\`a di Pavia\\
	via Ferrata 1, 27100 Pavia, Italy}
	\email{valeria.marcucci@unipv.it}

\thanks{This work has been partially supported by 1) FAR 2010 (PV) \emph{"Variet\`a algebriche, calcolo algebrico, grafi orientati e topologici"}; 2) INdAM
(GNSAGA)}

\date{\today}

\newcommand{\query}[1]{\marginpar{%                         %%
\vskip-\baselineskip %raise the marginpar a bit             %%
\raggedright\footnotesize                                   %%
\itshape\hrule\smallskip#1\par\smallskip\hrule}}            %%
\newcommand{\removequeries}{\renewcommand{\query}[1]{}}    %%
\removequeries
\begin{document}

\begin{abstract}
We prove that the geometric genus $p$ of a curve in a very generic Jacobian of dimension $g>3$ satisfies either $p=g$ or  $p> 2g-3$. This gives a positive answer to a conjecture of Naranjo and Pirola. For low values of $g$ the second inequality can be further improved to $p>2g-2$.
\end{abstract}

\maketitle

%
%\tableofcontents
%\makeatletter
%\def\@tocwriteb#1#2#3{}
%\makeatother

\section{Introduction}

In this paper we deal with curves in abelian varieties and, more precisely, in Jacobian varieties. This topic is classically known as \emph{theory of correspondences}: the $\mathbb{Z}$-module of the equivalence classes of correspondences between two curves is, in fact, canonically isomorphic to the group of homomorphisms between their Jacobians (see e.g. \cite[Theorem 11.5.1]{BL}). In \cite{artPirola} it is proved that all curves of genus $g$ lying on a very generic Jacobian variety of dimension $g\geq 4$ are birationally equivalent. We recall that a Jacobian $J\pa{C}$ is said to be \emph{very generic} if $[J\pa{C}]$ lies outside a countable union of proper analytic subvarieties of the Jacobian locus.

%We give a positive answer to a conjecture stated in \cite{NarPir} on the possible genus of a curve on a Jacobian variety.

We give conditions on the possible genus of a curve on a Jacobian variety. Namely, we show:

\begin{theo}
Let $D$ be a curve lying on a very generic Jacobian variety $J\pa{C}$ of dimension greater or equal to $4$. Then the geometric genus $g\pa{D}$ satisfies one of the followings
\[
\text{(i) }\, g\pa{D}\geq 2g\pa{C}-2, \qquad \qquad \text{(ii) }\, g\pa{D}=g\pa{C}.
\]
\end{theo}

%When $g\pa{D}=g\pa{C}$, it is known that $D$ is birationally equivalent to $C$ (see \cite{artPirola}).

Theorem A gives a positive answer to a conjecture stated in \cite{NarPir}, where the authors prove an analogous statement for Prym varieties. An equivalent formulation of Theorem A is the following:

\begin{theoB}
Given two smooth projective curves $C$ and $D$, where $C$ is very generic, $g\pa{C}\geq 4$ and $g\pa{D}<2g\pa{C}-2$, then either the N\'eron-Severi group of $C\times D$ has rank two, or it has rank three and $C\simeq D$.
\end{theoB}

Theorem B is implied by the fact that, in the previous hypotheses, if $f\colon J\pa{D} \ra J\pa{C}$ is a surjective map, then $J\pa{D}$ is isomorphic to $J\pa{C}$ and $f$ is the multiplication by a non-zero integer $n$.

We briefly outline the strategy of the proof: first we factorize the map $f\colon J\pa{D}\ra J\pa{C}$ into a surjective map $g\colon J\pa{D}\ra B$ of abelian varieties with connected kernel and an isogeny $h\colon B \ra J\pa{C}$. Then we study independently the two maps by a degeneration argument. The key point is the analysis of the limit $f_0\colon J\pa{D_0}\ra J\pa{C_0}$ of $f$, when $C$ degenerates to the Jacobian of an irreducible stable curve with one node.

By a rigidity result (see \cite{xiao}), if we let $C_0$ vary by keeping its normalization fixed, then also the normalization of $D_0$ does not change. The comparison of the relations between the extension classes in $\pic^0(J(\tilde D_0))$ and $\pic^0(J(\tilde C_0))$ of the two generalized Jacobians shows that the image of the map $H^1\pa{J\pa{\m{D}_0}, \mathbb{Z}}\ra H^1\pa{J\pa{\m{C}_0}, \mathbb{Z}}$, between the cohomology groups, is $n H^1\pa{J\pa{\m{C}_0}, \mathbb{Z}}$ for some non-zero integer $n$ (see Section \ref{subsec:Comparison of the extension classes}). This argument is an adaptation of the proof for the case $g\pa{D}=g\pa{C}$ (see \cite{artPirola}), but our setting requires a more careful analysis of the relations between the extension classes. For example, it is not sufficient to consider only the limit $f_0$, but we have to take into account also the relations coming from the limit of the dual map $\hat f_0$.

The comparison of two independent degenerations of the previous type allows us to prove that $B\simeq J\pa{C}$ and $h\colon B\ra J\pa{C}$ is the multiplication by $n$ (see Section \ref{subsec:doubledeg}). To conclude the proof, we notice that, since $J\pa{C}$ is very generic, the polarization $\Xi$, induced by $J\pa{D}$ on $B\simeq J\pa{C}$, is an integral multiple of the standard principal polarization of $J\pa{C}$. The analysis of the behavior of the map $g\colon J\pa{D}\ra B\simeq J\pa{C}$ at the boundary shows that $\Xi$ is principal. From the irreducibility of $J\pa{D}$ it follows that $g$ is an isomorphism (see Section \ref{subsec:conclusion}).

It seems natural to suppose that strict inequality holds in case \emph{(i)} of Theorem A, that is: there are no curves of genus $2g\pa{C}-2$ on a very generic Jacobian $J\pa{C}$ of dimension $g\pa{C}\geq 4$ (see Conjecture \ref{conj}). However, the previous argument cannot be applied because it is no longer possible to use the rigidity result. In the last part of the paper we prove a weaker statement which supports our conjecture: given a map $f\colon J\pa{D} \ra J\pa{C}$, from a smooth curve of genus $2g\pa{C}-2$ to a very generic Jacobian of dimension $g\pa{C}\geq 4$, the induced map $\varphi\colon D \ra A$, where $A$ is the identity component of the kernel of $f$, is birational on the image (see Proposition \ref{prop:bir}). The proof is based on a result on Prym varieties stated in \cite{PrymVP}, which asserts that a very generic Prym variety, of dimension greater or equal to $4$, of a ramified double covering is not isogenous to a Jacobian variety. When $g\pa{C}=4$ or $g\pa{C}=5$, the analysis of the possible deformations of $D$ in $A$ shows that $\varphi$ is constant and thus that $f$ is the zero map.

\section{Notations and preliminaries}
We work over the field $\mathbb{C}$ of complex numbers.

Each time we have a family of objects parameterized by a scheme $X$ (respectively by a subset $Y\subset X$) we say that the \emph{generic} element of the family has a certain property $\mathfrak{p}$ if $\mathfrak{p}$ holds on a non-empty Zariski open subset of $X$ (respectively of $Y$). Moreover, we say that a \emph{very generic} element of $X$ (respectively of $Y$) has the property $\mathfrak{p}$ if $\mathfrak{p}$ holds on the complement of a union of countably many proper subvarieties of $X$ (respectively of $Y$).

We denote by $M_g$ the moduli space of smooth projective curves of genus $g$ and by $\bar M_g$ the Deligne-Mumford compactification of $M_g$, that is the moduli space of stable curves of genus $g$. Let $M_g^0$ be the open set of $M_g$ whose points correspond to curves with no non-trivial automorphisms and let $\bar M_g^0$ be the analogous open set in $\bar M_g$. We denote by $\delta_0$ the divisor of $\bar M_g$ whose generic point parameterizes the isomorphism class of an irreducible stable curve with one node.

Given a projective curve $C$, we denote by $g\pa{C}$ its geometric genus. Given an abelian variety $A$ we will denote sometimes by $\hat A$ its dual abelian variety $\pic^0\pa{A}$.

We recall that if $J$ is a very generic Jacobian, then $\rk\pa{\NS\pa{J}}=1$; in particular $J$ has no non-trivial abelian subvarieties (cf. \cite{baseN}).

We will also need the following results:

\begin{theorem}[{\cite[Remark 2.7]{xiao}}]
\label{teo:xiao}
Let $J$ be a very generic Jacobian of dimension $n\geq 2$, $D$ be a smooth projective curve and $f\colon D \ra J$ be a non-constant map. If $g\pa{D}<2n-1$, then the only deformations of $\pa{D,f}$, with $J$ fixed, are obtained by composing $f$ with translations.
\end{theorem}

\begin{cor}
\label{cor:xiao}
Let $J$ be a very generic Jacobian of dimension $n\geq 2$, $\m{D}$ be a family of smooth projective curves over a smooth connected scheme $B$ and
\[
F\colon \pa{J\times B}/B \ra J\pa{\m{D}}/B
\]
be a non-constant map of families of Jacobians. If
\[
\dim J\pa{\m{D}}-\dim B<2n-1,
\]
then $\m{D}$ is a trivial family.\query{una mappa tra var. abeliane è rigida, modulo traslazioni. reticolo in reticolo!}
\end{cor}

%\begin{theorem}[\cite{artPirola}]
%\label{teo:stessogenere}
%If $f\colon J' \ra J$ is an isogeny between two Jacobians of dimension $g\geq 4$ and $J$ is very generic, then $J\simeq J'$ and $f$ is the multiplication by an integer.
%\end{theorem}

\subsection{Semi-abelian varieties}
\label{subsec:semiabvar}

Throughout the paper we will use some properties of degenerations of abelian varieties. Standard references for this topic are \cite{FaltingsLibro} and \cite{ChaiLibro} or, for Jacobian varieties, \cite{SerreLibroJacobDeg, NeronModels, alexeevcomp, OdaSeshadri}. Here we recall some basic facts.

\begin{defin}
Given an abelian variety $A$, its \emph{Kummer variety} $\mathcal{K}\pa{A}$ is the quotient of $A$ by the involution $x \mapsto -x$. We denote by  $\mathcal{K}^0\pa{A}$ the Kummer variety of $\pic^0\pa{A}$.
\end{defin}

\begin{defin}
A \emph{semi-abelian variety} $S$ of rank $n$ is  an extension
\begin{equation}
\label{eq: varsemiab}
0 \ra T \ra S \ra A \ra 0
\end{equation}
of an abelian variety $A$ by an algebraic torus $T=\prod^n \mathbb{G}_m$.
\end{defin}
Note that \eqref{eq: varsemiab} induces a long exact sequence
\begin{equation*}
1 \ra  H^0\pa{A, \fascio{A}^*} \ra H^0\pa{S ,\fascio{S}^*} \ra H^0\pa{T, \fascio{T}^*} \xrightarrow{\delta} H^1\pa{A, \fascio{A}^*},
\end{equation*}
and thus a group homomorphism
\begin{equation}
\label{eq:homcar}
h\colon \mathbb{Z}^n\simeq\Hom\pa{T, \mathbb{G}_m}\ra \pic^0\pa{A}.
\end{equation}
Viceversa, each homomorphism $h$ determines an exact sequence as in \eqref{eq: varsemiab} (see \cite[Chapter II, Section 2]{ChaiLibro}). In particular, the classes of isomorphism of semi-abelian varieties of rank $1$ with compact part $A$ are parameterized (up to multiplication by $-1$) by the homomorphisms $\mathbb{Z} \ra \pic^0\pa{A}$ and, consequently, by the points of the Kummer variety $\mathcal{K}^0\pa{A}$. By the previous argument we have the following proposition.
\begin{prop}
\label{prop:mappasemiab}
Consider two semi-abelian varieties
\[
0 \ra T \ra S \ra A \ra 0, \qquad 0 \ra T' \ra S' \ra A' \ra 0.
\]
A map $f\colon S \ra S'$ is determined by a map of abelian varieties $g\colon A \ra A'$ and by a morphism of groups
\[
\chi\colon \Hom\pa{T', \mathbb{G}_m}\ra \Hom\pa{T, \mathbb{G}_m}
\]
such that the following diagram commutes
\[
\xymatrix{
\mathbb{Z}^m\simeq \Hom\pa{T', \mathbb{G}_m}\ar[d]_{\chi}\ar[r] & \pic^0\pa{A'}\ar[d]^{g^*}\\
\mathbb{Z}^n\simeq \Hom\pa{T, \mathbb{G}_m}\ar[r] & \pic^0\pa{A}
}
\]
\end{prop}

\begin{defin}
Let $D$ be a projective curve having only nodes (ordinary double points) as singularities. The \emph{generalized Jacobian variety} of $D$ is defined as $J\pa{D}:=\pic^0\pa{D}$.
\end{defin}

Notice that, if $C$ is the normalization of $D$, we have a surjective morphism
\[
\pic^0(D) \ra \pic^0\pa{C}
\]
whose kernel is an algebraic torus; thus $\pic^0\pa{D}$ is a semi-abelian variety.

We recall the following result (see e.g. \cite[Theorem 2.3]{alexeevcomp} or \cite[Proposition 10.2]{OdaSeshadri}):
\begin{prop}
\label{prop:jacob}
The group homomorphism \eqref{eq:homcar} corresponding to $J\pa{D}$ can be identified with a map
\[
h\colon H_1\pa{\Gamma, \mathbb{Z}} \ra \pic^0\pa{J\pa{C}},
\]
where $\Gamma$ is the (oriented) dual graph of $D$, defined as follows. Given $e=\sum_i n_i e_i\in H_1\pa{\Gamma, \mathbb{Z}}$, where $e_i$ is the edge of $\Gamma$ corresponding to the node $N_i$ of $D$, then $h\pa{e}=\sum n_i[p_i-q_i]$, where $p_i$ and $q_i$ are mapped to $N_i$ by the natural morphism $C\ra D$.
\end{prop}

%\begin{prop}
%\label{prop:jacob}
%Consider a non-singular projective curve $C$ and $2n$ distinct points $x_1,\ldots,x_n$, $y_1, \ldots, y_n$ on $C$. Let $D$ be the nodal curve obtained from $C$ by pinching $x_i$ and $y_i$ for $i=1,\ldots, n$. The semi-abelian variety $J\pa{D}$ is determined by the morphism
%\begin{align*}
%h\colon \mathbb{Z}^n&\ra \pic^0{J\pa{D}}\\
%\pa{\alpha_1,\ldots,\alpha_i,\ldots,\alpha_n} &\mapsto \sum_{i=1}^n\alpha_i[x_i-y_i].
%\end{align*}
%\end{prop}

Given a non-singular projective curve $C$, we denote by $C_{p,q}$ the nodal curve obtained from $C$ by pinching the points $p,q\in C$. By Proposition \ref{prop:jacob}, the semi-abelian variety $J\pa{C_{p,q}}$ is the extension of $J\pa{C}$ by $ \mathbb{G}_m$ determined by $\pm [p-q]\in \mathcal{K}^0\pa{J\pa{C}}$.

Given a smooth projective curve $C$ of genus greater than 1, we denote by $\Gamma_C$ the image of the difference map
\begin{align}
\label{eq:difference}
C\times C &\ra J\pa{C} \isom \pic^0\pa{J\pa{C}}\\
\notag \pa{p,q} &\mapsto [p-q].
\end{align}
The image $\Gamma'_C$ of $\Gamma_C$ through the projection $\sigma_{C}\colon \pic^0\pa{J\pa{C}} \ra \mathcal{K}^0\pa{J\pa{C}}$ is a surface, in the Kummer variety, that parameterizes generalized Jacobian varieties of rank $1$ with abelian part $J\pa{C}$. Given an integer $n\in \mathbb{Z}$, we denote by $n\Gamma_C$ the image of $\Gamma_C$ under the multiplication by $n$ and by  $n\Gamma'_C$ the projection of $n\Gamma_C$ in $\mathcal{K}^0\pa{J\pa{C}}$.

\begin{prop}
\label{prop:diffpern}
Let $C$ be a non-hyperelliptic curve and $n$ be a non-zero integer. Then
\begin{enumerate}
\item\label{eq:diff1} $n\Gamma_C$ is birational to $C\times C$.
\item\label{eq:diff2} $n\Gamma'_C$ is birational to the double symmetric product $C_2$ of $C$.
\end{enumerate}
In particular, $\Gamma_C$ is birational to $C\times C$ and $\Gamma'_C$ is birational to $C_2$.
\end{prop}
\begin{proof}
Arguing as in \cite[Lemma 3.1.1, Proposition 3.2.1]{artPirola}, we can assume $n=1$. First we prove \eqref{eq:diff1}: given a generic point $\pa{a,b}\in C\times C$, if there exists $\pa{c,d}\in C\times C$ such that $[a-b]=[c-d]$, then $C$ is hyperelliptic. Statement \eqref{eq:diff2} follows from \eqref{eq:diff1}.\query{grado non più di due, due si fa a mano}
\end{proof}

\begin{no}
Given an abelian (or a semi-abelian) variety $A$ and a non-zero integer $n$, we denote by $\mathfrak{m}_n\colon A \ra A$ the map $x \mapsto nx$.
\end{no}

\begin{prop}
\label{prop:pmenoq}
Let $C$ be a generic curve of genus $g\geq 3$ and $[a-b]$ be a generic point of $\Gamma_C$. If there exist a point $[c-d]\in \Gamma_C$ and two positive integers $m,n$ such that $n[a-b]=m[c-d]$, then  $n=m$ and $[a-b]=\pm[c-d]$.
\end{prop}
\begin{proof}
Assume, by contradiction, that for each $a,b\in C$ there are $c,d\in C$ such that $na+md\equiv mc+nb$. It follows that $C$ has a two-dimensional family of maps of degree $m+n$ to $\mathbb{P}^1$ with a ramification of order $n+m-2$ over two distinct points. By a count of moduli we get a contradiction.
\end{proof}

\section{Local Systems}
\label{sec:localsystems}

In this section we introduce the monodromy representation of a local system. A standard reference for this topic is \cite[Chapter 3]{VoisinII}.

\begin{defin}
Let $X$ be an arcwise connected and locally simply connected topological space and $G$ be an abelian group. A \emph{local system} $\m{G}$, of stalk $G$, on $X$ is a sheaf on $X$ which is locally isomorphic to the constant sheaf of stalk $G$.
\end{defin}

Given two local systems $\m{G}$ and $\m{F}$ on $X$, we say that $\varphi\colon \m{G}\ra \m{F}$ is a morphism of local systems if $\varphi$ is a map of sheaves.

We recall that (see \cite[Corollary 3.10]{VoisinII}), given a point $x\in X$, a local system $\m{G}$ induces a representation
\begin{align*}
\rho\colon \pi_1\pa{X,x} &\ra \Aut\pa{\m{G}_x}=\Aut\pa{G}\\
\gamma &\mapsto \rho_\gamma
\end{align*}
of the fundamental group that is called \emph{monodromy representation}. The group $\rho\pa{\pi_1\pa{X,x}}$ is the \emph{monodromy group}. It is possible to define a functor, the \emph{monodromy functor}, from the category of local systems on $X$ to the category of abelian representations of $\pi_1\pa{X,x}$, which associates $\rho$ to $\m{G}$. It holds:

\begin{prop}
The monodromy functor induces an equivalence of categories between the category of local systems on $X$ and the category of abelian groups with an action on $\pi_1\pa{X,x}$.
\end{prop}

\begin{rmk}
\label{rmk:monodromiaaperto}
We recall that, given a local system $\m{G}$ on $X$, a map $\phi \colon Y \ra X$ and a point $y\in Y$, the monodromy representation of $\pi_1\pa{Y,y}$ corresponding to the local system $\phi^{-1}\m{G}$ is the composition of the natural morphism $\pi_1\pa{Y,y} \ra \pi_1\pa{X,\phi\pa{y}}$ with the monodromy representation of $\pi_1\pa{X,\phi\pa{y}}$.
\end{rmk}

In the following we assume $G$ to be a \emph{lattice}, that is an abelian finitely generated free group of even rank. The \emph{rank of the local system} $\m{G}$ is the rank of the group $G$.

\begin{defin}
Let $G\simeq \mathbb{Z}^{2g}$ be a lattice of rank $2g$. A \emph{polarization} of $G$ is a non-degenerate, antisymmetric, bilinear form $\theta\colon G\times G \ra \mathbb{Z}$. If the induced map $G \ra \Hom\pa{G, \mathbb{Z}}$ is an isomorphism, we say that $\theta$ is a \emph{principal polarization}.
\end{defin}

Given a principal polarization $\theta$ of $G$, a \emph{symplectic basis} for $G$ (with respect to $\theta$) is a minimal system of generators $a_1,\ldots,a_g,b_1,\ldots,b_g$ such that
\[
\theta\pa{a_i,b_i}=1, \qquad \theta\pa{a_i,a_j}=0, \qquad  \theta\pa{b_i,b_j}=0, \qquad \theta\pa{a_i,b_j}=0, \qquad \forall i, \forall j\neq i.
\]

We denote by $\Aut\pa{G, \theta}$ the group of \emph{symplectic automorphisms} of $G$, that are the automorphisms $T\colon G\ra G$ satisfying
\[
\theta\pa{a,b}=\theta\pa{T\pa{a},T\pa{b}}\qquad \forall a,b\in G.
\]
Let $g\in G$; the map $T_g\colon G \ra G$, defined as
\begin{equation}
\label{eq:autopicardlef}
T_g\pa{a}=a+ \theta\pa{g, a} g \qquad \forall a \in G,
\end{equation}
is a symplectic automorphism of $G$. We notice that, for each $N\in \mathbb{Z}$, we have $T_g^N=T_{Ng}$; in particular, $T_g$ has finite order in $\Aut\pa{G, \theta}$ if and only if $g=0$ and $T_g=\Id_G$.

\begin{defin}
Let $\m{G}$ be a local system on $X$. A \emph{polarization} (resp. a \emph{principal polarization}) $\Theta$ of $\m{G}$ is a map of local systems $\Theta\colon \m{G}\times \m{G} \ra \m{Z}$, where $\m{Z}$ is the constant sheaf $\mathbb{Z}$, such that, for each $x\in X$, $\Theta_x\colon \m{G}_x\times \m{G}_x\ra \mathbb{Z}$ is a polarization (resp. a principal polarization) of $\m{G}_x$. \query{oppure sez continua $\wedge^2 \m{G}$}
\end{defin}

Now we state a result that will be useful later in the paper.

\begin{no}
Given a lattice automorphism $T\colon G \ra G$, we denote by
\[
\Inv\pa{T}:=\set{x\in G \mbox{ s.t. } T\pa{x}=x}
\]
the subgroup of the elements of $G$ that are fixed by $T$.
\end{no}

\begin{prop}
\label{prop:monodromydoubledegen}
Let $\m{H} \hookrightarrow \m{G}$ be an injective map of local systems on $X$ of the same rank. Given $x\in X$, denote by $\rho\colon \pi_1\pa{X,x} \ra \m{H}_x$ the monodromy representation associated to $\m{H}$ and by $\sigma\colon \pi_1\pa{X,x} \ra \m{G}_x$ the monodromy representation associated to $\m{G}$. Consider two elements $\gamma_1, \gamma_2 \in \pi_1\pa{X,x}$ and set
\[
G_i:=\Inv\pa{\sigma_{\gamma_i}},\qquad H_i:=\Inv\pa{\rho_{\gamma_i}}\qquad \forall \, i=1,2.
\]
%\[
%G_i:=\set{g\in \m{G}_x \text{ s.t. } \rho_{\gamma_i}\pa{g}=g},\qquad H_i:=\set{h\in \m{H}_x \text{ s.t. } \rho_{\gamma_i}\pa{h}=h}.
%\]
Assume that
\begin{enumerate}
\item $G_1+G_2= \m{G}_x$;
\item $G_1\cap G_2\neq \set{0}$;
\item for each $i=1,2$, $H_i= n_iG_i$ for some $n_i\in \mathbb{N}$.
\end{enumerate}
Then $n_1=n_2$ and $\m{H}_x=n_1\m{G}_x$.
\end{prop}
\begin{proof}
Let $a\in G_1\cap G_2$ such that $a$ is not zero and it is not a multiple. It holds
\[
n_2a\in G_1 \cap H_2 \subset G_1 \cap \m{H}_x= H_1 =n_1G_1,
\]
and consequently $n_1\leq n_2$. In the same way, we find $n_2\leq n_1$. For the second statement, observe that
\[
\m{H}_x=H_1+H_2=n_1G_1+n_1G_2=n_1\m{G}_x.
\]
\end{proof}

\subsection{Family of curves}
\label{subsec:familyofcurves}
We conclude this section by recalling that, given a holomorphic, submersive and projective morphism $\phi\colon Y \ra X$ between complex manifolds, $R^k\phi_* \mathbb{Z}$ is a local system on $X$ and the monodromy representation on the cohomology of the fibre $H^*\pa{Y_x, \mathbb{Z}}$ is compatible with the cup product (see \cite[Section 3.1.2]{VoisinII}).
 In particular, if we consider a family of curves $\omega\colon \m{C} \ra X$ and the corresponding family of Jacobian varieties $\phi \colon J\pa{\m{C}}\ra X$, then  $R^1\phi_* \mathbb{Z}\simeq \pa{R^1\omega_* \mathbb{Z}}^*$ has a naturally polarization $\pairing$ induced by the intersection form.
 %%%and the monodromy representation associated to this local system is compatible with the polarization.

Let $Y$ be a complex surface and $\phi\colon Y \ra \Delta$ be a \emph{Lefschetz degeneration}. This means that $\phi$ is a holomorphic projective morphism with non-zero differential over the punctured disk $\Delta^*$ and that the fibre $Y_0$ is an irreducible curve with an ordinary double point as its only singularity (see \cite[Section 3.2.1]{VoisinII}). Given a point $x\in \Delta^*$, the \emph{vanishing cocycle} $\delta\in H^1\pa{Y_x, \mathbb{Z}}$ is a generator for
\[
\ker\pa{H^1\pa{Y_x , \mathbb{Z}}\simeq H_1\pa{Y_x, \mathbb{Z}}\ra H_1\pa{Y, \mathbb{Z}}}.
\]
We recall the following fact (see \cite[Theorem 3.16]{VoisinII}).

\begin{prop}[Picard-Lefshetz formula]
\label{prop:picardlefsch}
The image of the monodromy representation
\begin{equation*}
\pi_1\pa{\Delta^*, x} \ra \Aut\pa{H^1\pa{Y_x, \mathbb{Z}}, \pairing}
\end{equation*}
is generated by an element $T_\delta$ defined as in \eqref{eq:autopicardlef}. In particular, given a symplectic basis $a_1,\ldots,a_g,b_1, \ldots b_g$ of $H^1\pa{Y_x, \mathbb{Z}}$ such that $a_1=\delta$, then
\[
\Inv\pa{T_\delta}:=\set{a\in H^1\pa{Y_x, \mathbb{Z}} \text{ s.t. } T_\delta\pa{a}=a}
\]
is generated by $a_2, \ldots, a_g,b_1,\ldots,b_g$.
\end{prop}

\begin{rmk}
\label{rmk:piclef}
Notice that, after a finite base change $z\mapsto z^k$ of $\Delta$, the new generator of the monodromy group is $T_{k\delta}$.
\end{rmk}

\section{Main theorem}
\label{sec:mainTheorem}
The present section is devoted to prove the following result.

\begin{theorem}
\label{teo:main}
Let $J$ be a very generic Jacobian variety of dimension $g\geq 4$ and $\Omega$ be a curve lying on $J$. Then either $g\pa{\Omega}\geq 2g-2$ or $g\pa{\Omega}=g$.\query{se è very generic è semplice}
\end{theorem}

Assume that for a very generic Jacobian variety of dimension $g\geq 4$ there exists a curve of genus $p<2g-2$ lying on it. We want to prove that $p=g$. Arguing as in \cite{artPirola}, one finds a map of families
\begin{equation}
\label{eq:mappaF}
F\colon J\pa{\m{D}}/\m{V} \ra J\pa{\m{C}}/\m{V},
\end{equation}
verifying:
\begin{itemize}
\item $\m{V}$ is a finite étale covering of a dense open subset $\m{U}$ of $M^0_g$;
\item for all $t\in \m{V}$, $J\pa{\m{C}_t}$ is a Jacobian variety of dimension $g$;
\item for all $t\in \m{V}$, $J\pa{\m{D}_t}$ is a Jacobian variety of dimension $p<2g-2$;
\item $F$ is surjective on each fibre.
\end{itemize}
\query{ $F_t$ is non-constant. We can also assume that the induced map of abelian schemes $J\pa{F}\colon J\pa{\m{D}} \ra J\pa{\m{C}}$ is surjective rk che cala: condizione chiusa!}
There exists a family of abelian varieties $\m{B}/\m{V}$ such that $F$ can be factorized in two morphisms of abelian schemes
\begin{equation}
\label{eq:fattorizzazione}
G\colon J\pa{\m{D}} \ra \m{B}, \qquad H\colon \m{B} \ra J\pa{\m{C}},
\end{equation}
where $G$ is surjective with connected fibres and $H$ is finite (this can be also seen as a simple application of the Stein factorization, see e.g. \cite[Chapter III, Corollary 11.5]{HART}\query{EGAIII 4.3.3}). The kernel $\m{A}$ of $G$ is a family of abelian varieties and, for each $t\in \m{V}$, $\m{A}_t$ is the identity component of $\ker \pa{F_t\colon J\pa{\m{D}_t}\ra J\pa{\m{C}_t}}$. In the following we prove that $\m{B}\simeq J\pa{\m{C}}$, $H\colon \m{B} \ra J\pa{\m{C}}$ is the multiplication by a non-zero integer (Corollary \ref{cor:moltn}) and $G\colon J\pa{\m{D}} \ra \m{B}$ is an isomorphism (Section \ref{subsec:conclusion}).

Before starting with the proof of Theorem \ref{teo:main} in the general case, we notice that, for $g=4$, the statement is a direct consequence of the following proposition.

\begin{prop}
\label{prop:g+1}
There are no curves of genus $g+1$ lying on a very generic Jacobian of dimension $g\geq 4$.
\end{prop}
\begin{proof}
With the previous notation, observe that, if $p=g+1$, $\m{A}$ is a family of elliptic curves. Furthermore, for each $t\in \m{V}$, the map $J\pa{\m{D}_t} \ra \m{A}_t$ is induced by a non-constant map of curves $\m{D}_t \ra \m{A}_t$. Since the moduli space of coverings of genus $g+1$ of elliptic curves has dimension $2g$, it follows that $\m{A}\ra J\pa{\m{D}}$ is constant on a dense open set of $\m{V}$ and we get a contradiction.\query{assurdo perché le classi di isogenia sono numerabili}
\end{proof}

%From now on we will assume $g\geq 5$. The argument works also for $g=4$ but, with the previous assumption, we can apply Theorem \ref{teo:stessogenere} and the proof is simpler.

\subsection{Comparison of the extension classes}
\label{subsec:Comparison of the extension classes}
Let $C$ be a very generic smooth curve of genus $g-1$, in particular assume that $J\pa{C}$ is simple. Given a generic point $[p-q]\in \Gamma'_C$ (see \eqref{eq:difference} in Section \ref{subsec:semiabvar}), consider the nodal curve $C_{p,q}$ obtained  from $C$ by pinching $p$ and $q$ and a non-constant map $\tau\colon \Delta \ra \bar M^0_g$ such that $\tau\pa{\Delta^*}\subset M^0_g$ and $\tau\pa{0}$ is the class of $C_{p,q}$.\query{$\m{V}$ è denso in $M^0_g$ , $\m{V}$ denso in $M^0_g$, ogni punto a un intorno che interseca $\m{V}$} Let us restrict our initial families of Jacobian varieties $J\pa{\m{D}}$ and $J\pa{\m{C}}$, defined in \eqref{eq:mappaF}, to $\Delta^*$ (we suppose that $\tau\pa{\Delta^*}\subset \m{U}$). Changing base, if necessary, we get a map of families of semi-abelian varieties
\[
F\colon J\pa{\m{D}}/\Delta \ra J\pa{\m{C}}/\Delta,
\]
satisfying the following conditions:
\begin{itemize}
\item $F|_{\Delta^*}$ coincides with the map defined in \eqref{eq:mappaF};
\item $\m{C}_0=C_{p,q}$;
\item $\m{D}_0$ is a nodal curve of arithmetic genus $p<2g-2$.
\end{itemize}

Given the map of families
\[
\hat F\colon J\pa{\m{C}}/\Delta\ra J\pa{\m{D}}/\Delta,
\]
obtained from $F$ by dualization, we denote by $f\colon J\pa{C} \ra J(\tilde {\m{D}_0})$ the map induced by $\hat F_0$ on the abelian quotients.

The aim of the following propositions is to describe the limit $J\pa{\m{D}_0}$ of the family of Jacobian varieties $J\pa{\m{D}}$ and the map $\hat F_0\colon J\pa{\m{C}_0}\ra J\pa{\m{D}_0}$.

\begin{prop}
\label{prop:norm}
The smooth curve $C$ is isomorphic to a connected component of $\tilde {\m{D}_0}$ and $f=i \circ \mathfrak{m}_n$, where $\mathfrak{m}_n\colon J\pa{C} \ra J\pa{C}$ is the multiplication by a non-zero integer $n$ and $i\colon J\pa{C}\ra J(\tilde {\m{D}_0})$ is the natural inclusion.
\end{prop}
\begin{proof}
Since $J\pa{C}$ is simple, $f$ has finite kernel and $f\pa{J\pa{C}}$ is an irreducible abelian subvariety, of dimension $g-1$, that does not contain curves of geometric genus lower than $g-1$. From the inequality $g\pa{\m{D}_0}\leq p-1<2g-3$, we can conclude that there is only a connected component $X$ of $\tilde {\m{D}_0}$ such that $g\pa{X}\geq g-1$. It follows $f\pa{J\pa{C}}\subset J\pa{X}$.
We want to prove that $X$ is isomorphic to $C$ and that
\begin{equation}
\label{eq:suiquoz}
f\colon J\pa{C}\ra J\pa{X}
\end{equation}
is the multiplication by a non-zero integer $n$.
Now we procede by steps:

\medskip

\emph{Step I: rigidity.}

\noindent By varying the point $[p-q]$ in $\Gamma'_C$, and consequently the curve $C_{p,q}$, we can perform different degenerations of the family of curves $\m{D}/\m{V}$ (defined in \eqref{eq:mappaF}). By the previous construction, to each family of degenerations corresponds a non-constant map of abelian schemes
\[
\pa{J\pa{C}\times B}/B \ra J\pa{\m{X}}/B,
\]
where $J\pa{\m{X}}$ is a family of Jacobians and
\[
g-1=g\pa{C}\leq g\pa{X_t}< 2g\pa{C}-1=2g-3
\]
for each $t\in B$. By Corollary \ref{cor:xiao}, the family $J\pa{\m{X}}$ is trivial. This means that, though the limit $F_0\colon J\pa{\m{D}_0} \ra J\pa{\m{C}_0}=J\pa{C_{p,q}}$ depends on $[p-q]\in \Gamma'_C$, the map $f\colon J\pa{C} \ra J(X)$ on the abelian parts is independent of the choice of the point $[p-q]$.

\medskip

\emph{Step II: $f^*\pa{\Gamma_X}=n\Gamma_C$ for some non-zero integer $n$.}

\noindent Since $f$ is limit of a map of Jacobian varieties, by Proposition \ref{prop:mappasemiab} and Proposition \ref{prop:jacob}, for each $[p-q]\in \Gamma_C\subset \pic^0\pa{J\pa{C}}$ there exists $[y-z]\in \Gamma_X\subset \pic^0\pa{J\pa{X}}$ and a non-zero integer $n$ such that $f^*\pa{[y-z]}=n[p-q]$. It follows that
\[
\Gamma_C\subset \bigcup_{n\geq 1}\mathfrak{m}_n^{-1}\pa{f^*\pa{\Gamma_X}}.
\]
This implies that $\Gamma_C$ is contained in $\mathfrak{m}_n^{-1}\pa{f^*\pa{\Gamma_X}}$ for some $n$. By the irreducibility of $f^*\pa{\Gamma_X}$, it holds $f^*\pa{\Gamma_X}=n\Gamma_C$.

\medskip

\emph{Step III: $X$ is isomorphic to $C$ and $f\colon J\pa{C} \ra J\pa{X}$ is the multiplication by $n$.}

\noindent By the previous step and Proposition \ref{prop:diffpern}, we can define a rational dominant map
\[
X\times X \dashrightarrow \Gamma_X \xrightarrow{{f}^*} n\Gamma_C \dashrightarrow C\times C.
\]
Consequently we have a non-constant morphism $h\colon X \ra C$. By Riemann-Hurwitz formula we get
\[
4g-6>2p-2> 2g\pa{X}-2\geq \deg h\pa{2g\pa{C}-2}=\deg h\pa{2g-4}
\]
and so $\deg h=1$ and $C\simeq X$. Since we can assume $\End\pa{J\pa{C}}=\mathbb{Z}$ (see e.g. \cite{Parameters}), $f=\mathfrak{m}_k$ for some $k\in \mathbb{Z}$ (notice that, by Proposition \ref{prop:pmenoq}, $n=\pm k$).
\end{proof}

Denote by $D_1,\ldots,D_k$ the connected components of $\tilde {\m{D}_0}$. By Proposition \ref{prop:norm}, since $C$ has no non-trivial automorphisms, we can assume $D_1= C$. We have the following result:

\begin{prop}
\label{prop:D0}
The map of semi-abelian varieties $\hat F_0\colon J\pa{\m{C}_0}\ra J\pa{\m{D}_0}$ corresponds to the following morphism of extensions
\[
\xymatrix{
0 \ar[r] & \mathbb{G}_m \ar[r]\ar[d]^{\chi} & J\pa{\m{C}_0}\ar[d]^{\hat F_0} \ar[r] & J\pa{C} \ar[r]\ar[d]^{\eta} & 0\\
0 \ar[r] & \prod \mathbb{G}_m \ar[r] & J\pa{\m{D}_0} \ar[r] & J\pa{C}\times J\pa{D_2}\times \ldots\times J\pa{D_k}\ar[r] &0
}
\]
where $\chi\pa{1}=\pa{n,\ldots,0}$ and $\eta$ is the composition of the multiplication by $n$ with the inclusion in the first factor.
\end{prop}
\begin{proof}
We recall that (see Section \ref{subsec:semiabvar}) the semi-abelian variety $J\pa{\m{D}_0}$ is determined by a morphism
\[
\alpha \colon \mathbb{Z}^{m} \simeq \Hom\pa{\prod \mathbb{G}_m, \mathbb{G}_m }\ra \pic^0\pa{J\pa{C}}\times \pic^0\pa{J\pa{D_2}}\times \cdots \times \pic^0\pa{J\pa{D_k}},
\]
where $m:=p-g(\tilde {\m{D}_0})$. By Proposition \ref{prop:mappasemiab}, $\hat F_0\colon J\pa{\m{C}_0}=J\pa{C_{p,q}}\ra J\pa{\m{D}_0}$ corresponds to a diagram
\[
\xymatrix{
\mathbb{Z}^{m}\ar[d]_{\chi}\ar[r]^{\alpha\qquad\qquad\qquad\qquad\qquad} & \pic^0\pa{J\pa{C}}\times \pic^0\pa{J\pa{D_2}}\times \cdots \times \pic^0\pa{J\pa{D_k}}\ar[d]^{f^*}\\
\mathbb{Z}\ar[r]^{\beta} & \pic^0\pa{J\pa{C}}
}
\]
where, $f^*$ is the composition of the first projection with $\mathfrak{m}_n$ (see Proposition \ref{prop:norm}). Given a basis $e_1,\ldots, e_m$ of $\mathbb{Z}^m$, we claim that, for each $i=1,\ldots, m$, either $\chi\pa{e_i}=\pm n$, or $\chi\pa{e_i}=0$. It holds
\[
f^*\pa{\alpha\pa{e_i}}=\beta\pa{\chi\pa{e_i}}=j[p-q],
\]
for some $j\in \mathbb{Z}$. Then, by Proposition \ref{prop:jacob}, if $j\neq 0$,
\[
n[a-b]= f^*\pa{\alpha\pa{e_i}}=j[p-q]
\]
for some $[a-b]\in \Gamma_C$. From Proposition \ref{prop:pmenoq} it follows
\[
[p-q]=\pm[a_i-b_i]
\]
and $j=n$. We recall that, since $\hat F_0$ is limit of injective maps of Jacobians, $\chi\not\equiv 0$. Up to a suitable choice of the basis $e_1,\ldots, e_m$, we can assume $\chi\pa{e_1}=n$ and $\chi\pa{e_i}=0$ for $i=2,\ldots,m$.
\end{proof}

\begin{cor}
\label{cor:cohomology}
The image of the map
\[
H^1\pa{J\pa{\m{D}_0}, \mathbb{Z}} \ra H^1\pa{J\pa{\m{C}_0}, \mathbb{Z}},
\]
induced by $\hat F_0\colon J\pa{\m{C}_0}\ra J\pa{\m{D}_0}$ on the cohomology groups, is $n H^1\pa{J\pa{\m{C}_0}, \mathbb{Z}}$.
\end{cor}

\subsection{Local systems}
\label{subsec:teomon}

Let consider again the families of abelian varieties $J\pa{D}/\m{V}$, $\m{B}/\m{V}$ and $J\pa{\m{C}}/\m{V}$, defined in \eqref{eq:mappaF}, and \eqref{eq:fattorizzazione} and denote by
\[
\mu\colon J\pa{\m{D}}\ra \m{V}, \qquad \varphi\colon \m{B} \ra \m{V}, \qquad \psi\colon J\pa{\m{C}} \ra \m{V}
\]
the projections on the base. Given a point $t\in \m{V}$, we denote by
\begin{align*}
\nu&\colon \pi_1\pa{\m{V}, t} \ra H^1\pa{J\pa{\m{D}_t}, \mathbb{Z}},\\
\rho&\colon \pi_1\pa{\m{V}, t} \ra H^1\pa{\m{B}_t, \mathbb{Z}},\\
\sigma&\colon \pi_1\pa{\m{V}, t} \ra H^1\pa{J\pa{\m{C}_t}, \mathbb{Z}}
\end{align*}
the monodromy representations corresponding to the local systems $R^1\mu_*\mathbb{Z}, R^1\varphi_*\mathbb{Z}$ and $R^1\psi_*\mathbb{Z}$. We recall that $\m{V}$ has a finite map $\pi\colon \m{V}\ra \m{U}$ on an open dense set of $M^0_g$ and, by construction (see Section \ref{subsec:Comparison of the extension classes}),
\begin{equation}
\label{eq:restrcurve}
R^1\psi_*\mathbb{Z}=\pi^{-1}R^1\omega_*\mathbb{Z}
\end{equation}
where $\omega\colon \m{C} \ra \m{V}$ is the restriction of the universal family of curves to $\m{V}$.

The map $F\colon J\pa{\m{D}}\ra J\pa{\m{C}}$ (or, more precisely, its dual $\hat F\colon J\pa{\m{C}}\ra J\pa{\m{D}}$) induces a map of local systems
\begin{equation}
\label{eq:mapls}
R^1\mu_*\mathbb{Z}\xrightarrow{\mathfrak{F}}  R^1\psi_*\mathbb{Z}
\end{equation}
which factorizes in a surjective map $\mathfrak{G}\colon R^1\mu_*\mathbb{Z}\ra R^1\varphi_*\mathbb{Z}$ followed by an injective map $\mathfrak{H}\colon R^1\varphi_*\mathbb{Z}\ra R^1\psi_*\mathbb{Z}$.

In this Section we restrict the local systems introduced before to the pointed disk $\Delta^*$ considered in Section \ref{subsec:Comparison of the extension classes}. We have the following result:

\begin{prop}
\label{prop:specMonodromy}
Given a generator $\gamma$ of $\pi_1\pa{\Delta^*, t}$, then
\[
\Inv\pa{\rho_{\gamma}}=n\Inv\pa{\sigma_\gamma}.
\]
\end{prop}
\begin{proof}
%Let $\omega\colon \m{C} \ra \Delta^*$ be the restriction of the universal family of curves to $\Delta^*$. By our initial construction, we have
%\[
%\pi^{-1}R^1\omega_*\mathbb{Z}=R^1\psi_*\mathbb{Z},
%\]
%where $\pi: \Delta^* \ra \Delta^*$ is a finite covering.
By Proposition \ref{prop:picardlefsch}, Remark \ref{rmk:piclef} and \eqref{eq:restrcurve}, it follows that
\[
\Inv\pa{\sigma_\gamma}=\Im\pa{\zeta\colon H^1\pa{J\pa{\m{C}_0}, \mathbb{Z}}=H^1\pa{J\pa{\m{C}}, \mathbb{Z}} \ra H^1\pa{J\pa{\m{C}_t}, \mathbb{Z}}},
\]
where $\zeta$ is the map induced by the inclusion $J\pa{\m{C}_t}\hookrightarrow J\pa{\m{C}}$. This implies
\begin{align*}
\Inv\pa{\rho_\gamma}&= \Im\pa{H^1\pa{\m{B}_t, \mathbb{Z}} \xrightarrow{\mathfrak{H}_t}  H^1\pa{J\pa{\m{C}_t}, \mathbb{Z}} } \cap \Inv\pa{\sigma_\gamma}\\&=
\Im\pa{H^1\pa{\m{D}_t, \mathbb{Z}} \xrightarrow{\mathfrak{F}_t}  H^1\pa{J\pa{\m{C}_t}, \mathbb{Z}} } \cap \Inv\pa{\sigma_\gamma}\\&=
\Im\pa{H^1\pa{\m{D}_t, \mathbb{Z}} \xrightarrow{\mathfrak{F}_t}  H^1\pa{J\pa{\m{C}_t}, \mathbb{Z}} } \cap \Im\pa{\zeta}\\&=
\Im \pa{H^1\pa{\m{D}_0, \mathbb{Z}} \ra  H^1\pa{J\pa{\m{C}_0}, \mathbb{Z}} \xrightarrow{\zeta}  H^1\pa{J\pa{\m{C}_t}, \mathbb{Z}}  }\\ &=n\Im\pa{ H^1\pa{J\pa{\m{C}_0}, \mathbb{Z}} \xrightarrow{\zeta} H^1\pa{J\pa{\m{C}_t}, \mathbb{Z}} }\\&=n\Inv\pa{\sigma_\gamma},
\end{align*}
where the next to last identity follows from Corollary \ref{cor:cohomology}.
\end{proof}

\subsection{Double degeneration}
\label{subsec:doubledeg}

Let $P$ be a very generic irreducible stable curve of arithmetic genus $g\geq 5$, with exactly two nodes $N_1, N_2$ as singularities. Consider an open analytic neighborhood $U\subset \bar M_g^0$ of $[P]\in \bar M_g$ biholomorphic to a $\pa{3g-3}$-dimensional polydisk. Assume that $U$ has local coordinates $z_1,\ldots,z_{3g-3}$ centered at $[P]$ and such that the local equation of $\delta_0\cap U$ is $z_1\cdot z_2=0$. Furthermore, for $i=1,2$, $z_i=0$ is the local equation in $U$ of the locus $\delta^{i}_0$ where the singularity $N_i$ persists. We set $U':=U\setminus \delta_0$.

Let consider our initial families of isogenies $H\colon \m{B}/\m{V}\ra J\pa{\m{C}}/\m{V}$, defined in \eqref{eq:fattorizzazione}. We recall that $\pi\colon \m{V}\ra \m{U}$ is a finite étale covering of a dense open subset $\m{U}$ of $M^0_g$. Set $U^*:=U'\cap \m{U}$, $V^*:=\pi^{-1}\pa{U^*}$ and consider the restriction
\[
H\colon\m{B}/V^*\ra J\pa{\m{C}}/V^*.
\]
of $H$ to $V^*$. Then it holds:

\query{non serve prendere componente connessa perché il gruppo fondamentale non è banale, più o meno ha due generatori, forse più complicato}

\begin{prop}
For each $t\in V^*$, the abelian variety $B_t$ is isomorphic to $J\pa{\m{C}_t}$ and the map $H_t\colon B_t\ra J\pa{C_t}$ is the multiplication by a non-zero integer $n$.
\end{prop}
\begin{proof}
%Denote by $\varphi\colon \m{B} \ra V^*$ and $\psi\colon J\pa{\m{C}} \ra V^*$ the projections of the two families on the base. The map $H$ induces an injective map of local systems
%\[
%\xymatrix{
%R^1\varphi_*\mathbb{Z}\ar[rr] \ar[rd] && R^1\psi_*\mathbb{Z} \ar[dl]\\
%&V^*&
%}
%\]
%Let $\rho\colon \pi_1\pa{V^*, t} \ra H^1\pa{\m{B}_t, \mathbb{Z}}$ be the monodromy representation corresponding to the local system $R^1\psi_*\mathbb{Z}$ and $\sigma\colon \pi_1\pa{V^*, t} \ra H^1\pa{J\pa{\m{C}_t}, \mathbb{Z}}$ be the monodromy representation corresponding to the local system $R^1\varphi_*\mathbb{Z}$.
Let us restrict the local systems $R^1\varphi_*\mathbb{Z}$ and $R^1\psi_*\mathbb{Z}$, defined in Section \ref{subsec:teomon}, and their monodromy representations $\rho$ and $\sigma$, to $V^*$. The statement will follow from Proposition \ref{prop:monodromydoubledegen}.

Let $\omega\colon \m{C} \ra U'$ be the restriction of the universal family of curves to $U'$ and denote by $\ell$ the inclusion $U^* \hookrightarrow U'$. It holds $\pi^{-1}\ell^{-1}R^1\omega_*\mathbb{Z}=R^1\psi_*\mathbb{Z}$, where we recall that $\pi$ is the finite covering $\pi\colon V^* \ra U^*$. Set $u:=\pi\pa{t}$ and denote by $M\leq \Aut\pa{H^1\pa{J\pa{\m{C}_t}, \mathbb{Z}}}$ the monodromy group of the local system $R^1\psi_*\mathbb{Z}$ and by $L\leq \Aut\pa{H^1\pa{\m{C}_u, \mathbb{Z}}}=\Aut\pa{H^1\pa{J\pa{\m{C}_t}, \mathbb{Z}}}$ the monodromy group of the local system $R^1\omega_*\mathbb{Z}$. By Remark \ref{rmk:monodromiaaperto}, since $\pi_1\pa{U^*,u}\ra \pi_1\pa{U',u}$ is surjective, $M$ is a subgroup of finite index of $L$.

We consider a symplectic basis $\set{a_1,\ldots,a_g, b_1, \ldots, b_g}$ of $H^1\pa{\m{C}_u, \mathbb{Z}}$ such that, for $i=1,2$, $a_i$ is the vanishing cocycle for the Lefschetz degeneration centered in a point of $\delta^{i}_0$. We recall that, by Proposition \ref{prop:picardlefsch}, $L$ is generated by $T_{a_1}$ and $T_{a_2}$ (see \eqref{eq:autopicardlef} in Section \ref{sec:localsystems}). This implies that, for each $i=1,2$, there is an element $\gamma_i\in \pi_1\pa{V^*,t}$, such that $\sigma_{\gamma_i}=T_{k_ia_i}$ for some non-zero $k_i\in \mathbb{N}$. Thus, $\Inv\pa{\sigma_{\gamma_i}}$ is generated by $\set{a_1,\ldots,a_g, b_1, \ldots, b_g}\setminus \set{a_i}$ and, by Proposition \ref{prop:specMonodromy},
\[
\Inv\pa{\rho_{\gamma_i}}=n_i\Inv\pa{\sigma_{\gamma_i}},
\]
for some non-zero $n_i\in\mathbb{N}$. Proposition \ref{prop:monodromydoubledegen} implies $n_1=n_2$ and $H^1\pa{\m{B}_t, \mathbb{Z}}=n_1 H^1\pa{J\pa{\m{C}_t}, \mathbb{Z}}$.
\end{proof}

By considering a covering of $M^0_g$ of analytic open sets as $U'$, we have the following corollary.

\begin{cor}
\label{cor:moltn}
The abelian scheme $\m{B}/\m{V}$ is isomorphic to $J\pa{\m{C}}/\m{V}$ and the map $H\colon \m{B}/\m{V} \ra J\pa{\m{C}}/\m{V}$ is on each fibre the multiplication by a non-zero integer $n$.
\end{cor}

\subsection{Conclusion of the proof of Theorem \ref{teo:main}}
\label{subsec:conclusion}

Let consider the surjective map of abelian schemes
\[
G\colon J\pa{\m{D}}/\m{V}\ra \m{B}/\m{V}
\]
defined in \eqref{eq:fattorizzazione}. The aim of this section is to show that $G$ is an isomorphism. This concludes the proof of Theorem \ref{teo:main}.

\begin{rmk}
By Corollary \ref{cor:moltn}, $\ker G= 0$ implies that $J\pa{\m{D}}$ is isomorphic to $J\pa{\m{C}}$ and $F\colon J\pa{\m{D}} \ra J\pa{\m{C}}$ is the multiplication by $n$. We recall that, given a very generic point $t\in \m{V}$, we can assume $\NS\pa{J\pa{\m{C}_t}}=\mathbb{Z}$. This implies that $J\pa{\m{D}_t}$ is isomorphic to $J\pa{\m{C}_t}$ as a principally polarized abelian variety and, by Torelli theorem (see \cite{andreotti}) $\m{D}_t\simeq \m{C}_t$. It follows that the two families of curves $\m{D}$ and $\m{C}$ are isomorphic. Thus, when $g\pa{\m{D}_t}=g\pa{\m{C}_t}$, we recover the result in \cite{artPirola}.
\end{rmk}

The surjective map $G\colon J\pa{\m{D}}\ra \m{B}\simeq J\pa{C}$ induces an injective map of linear systems $R^1\psi_*\mathbb{Z} \hookrightarrow R^1\mu_*\mathbb{Z}$ (cf. Section \ref{subsec:teomon}). Thus $R^1\psi_*\mathbb{Z}$ has two different polarizations: the natural polarization $\Theta$ induced by the intersection of cycles on $\m{C}$, and the polarization $\Xi$ inherited by $R^1\mu_*\mathbb{Z}$ endowed with the polarization induced by the intersection of cycles on $\m{D}$ (see Section \ref{subsec:familyofcurves}).

In the following proposition we show that the two polarizations coincide. This implies that, given $t\in \m{V}$, the exact sequence of polarized abelian varieties
\[
0\ra J\pa{\m{C}_t}\ra J\pa{\m{D}_t}\ra J\pa{\m{D}_t}/J\pa{\m{C}_t}\ra 0,
\]
induced by the dual map $\hat G_t\colon J\pa{\m{C}_t}\ra J\pa{\m{D}_t}$, splits (see e.g. \cite[Corollary 5.3.13]{BL}). Since $\m{D}_t$ is an irreducible curve, the theta divisor of $J\pa{\m{D}_t}$ is irreducible and $\ker G_t=0$.

\begin{prop}
The polarizations $\Theta$ and $\Xi$ of $R^1\psi_*\mathbb{Z}$ coincide.
\end{prop}
\begin{proof}
Let $t$ be a very generic point in $\m{V}$ and let $\set{a_1,\ldots, a_g,b_1,\ldots,b_g}$ be a basis of $H^1\pa{J\pa{\m{C}_t},\mathbb{Z}}$ that is symplectic with respect to $\Theta_t$. Consider a Lefschetz degeneration $\m{C}/\Delta$ of $\m{C}_t$ such that $a_1$ is the vanishing cocycle. We restrict $J\pa{\m{D}}$ to $\Delta^*$ and, arguing as in Section \ref{subsec:Comparison of the extension classes}, up to a base change of $\Delta$, we get the limit
\[
G\colon J\pa{\m{D}}/\Delta \ra J\pa{\m{C}}/\Delta
\]
of the map $G$ when $\m{C}_t$ degenerates to a nodal curve. Notice that, by Corollary \ref{cor:moltn}, the map of semi-abelian varieties $F_0\colon J\pa{\m{D}_0}\ra J\pa{\m{C}_0}$ is the composition of the map $G_0\colon J\pa{\m{D}_0}\ra J\pa{\m{C}_0}$ with the multiplication by $n$.

Through the natural map
\[
H^1(\tilde{\m{C}_0},\mathbb{Z}) \hookrightarrow H^1\pa{\m{C}_0,\mathbb{Z}} = H^1\pa{\m{C},\mathbb{Z}} \ra H^1\pa{\m{C}_t,\mathbb{Z}}
\]
we identify $H^1(\tilde{\m{C}_0},\mathbb{Z})$ with the sub-lattice $L$ of $H^1\pa{J\pa{\m{C}_t},\mathbb{Z}}$ generated by the elements $\set{a_2,\ldots, a_g,b_2,\ldots,b_g}$. In this way, $\Theta_t|_{L}$ is the polarization induced by the intersection of cycles on $\tilde{\m{C}_0}$ and $\Xi_t|_{L}$ is the polarization inherited by $H^1(\tilde{\m{D}_0},\mathbb{Z})$ (endowed with the polarization induced by the intersection of cycles on $\tilde{\m{D}_0}$) through the inclusion $H^1(\tilde{\m{C}_0},\mathbb{Z}) \hookrightarrow H^1(\tilde{\m{D}_0},\mathbb{Z})$. By Proposition \ref{prop:norm},
\[
J(\tilde{\m{D}_0})\simeq J\pa{C}\times J\pa{D_2}\times\ldots\times J\pa{D_k}
\]
and the map $J(\tilde{\m{D}_0})\ra J\pa{\tilde{\m{C}_0}}$, induced by $G_0\colon J({\m{D}_0})\ra J\pa{{\m{C}_0}}$ on the compact quotient, is the first projection. Thus $\Theta_t|_{L}=\Xi_t|_{L}$. For a very generic $t\in \m{V}$, we can assume $\NS\pa{J\pa{C_t}}=\mathbb{Z}$ and $\Xi_t=n\Theta_t$ for some $n\in \mathbb{N}$. This implies $\Theta_t=\Xi_t$.
\end{proof}

\section{Curves of genus $2g-2$ on a Jacobian of dimension $g$.}

It is natural to ask whether the bound given in Theorem \ref{teo:main} is sharp. Notice that, if $p>2g-2$, it is possible to give examples of non-constant maps $\Omega \ra J$, from a curve of genus $p$ to a Jacobian of dimension $g$. Namely, it is always possible to find a finite covering of genus $p$ of a curve of genus $g$.
%If $p\leq 2g-2$, to give a non-constant map $\Omega \ra J$ is equivalent to give a curve lying on $J$ of geometric genus $g'\leq p$.

The following conjecture was suggested to us by Gian Pietro Pirola.

\begin{conj}
\label{conj}
There are no curves of geometric genus $2g-2$ lying on a very generic Jacobian of dimension $g\geq 4$.
\end{conj}

If the conjecture would be false, as in the previous case (see \eqref{eq:mappaF} in Section \ref{sec:mainTheorem}), we would find a map of families
\begin{equation*}
F\colon J\pa{\m{D}}/\m{V} \ra J\pa{\m{C}}/\m{V},
\end{equation*}
and a family of abelian varieties $\m{B}/\m{V}$ such that $F$ can be factorized in two morphisms (see \eqref{eq:fattorizzazione})
\begin{equation*}
G\colon J\pa{\m{D}} \ra \m{B} \qquad H\colon \m{B} \ra J\pa{\m{C}},
\end{equation*}
where $G$ has connected fibres and $H$ is a finite morphism. In this case, the abelian scheme $\m{A}:=\ker G$ would be a family of abelian varieties of dimension $g-2$ with a natural inclusion $I\colon \m{A} \hookrightarrow J\pa{\m{D}}$.

Let consider the dual map
\[
\hat I \colon J\pa{\m{D}}\simeq \hat {J\pa{\m{D}}} \ra \hat {\m{A}}
\]
and its composition with the Abel map
\begin{equation}
\label{eq:L}
L\colon \m{D} \ra \hat {\m{A}}.
\end{equation}
\begin{prop}
\label{prop:bir}
For a generic $t\in \m{V}$, the map $L_t\colon \m{D}_t \ra \hat {\m{A}}_t$ is birational on the image.
\end{prop}
\begin{proof}
The map of curves $L_t\colon \m{D}_t \ra W_t$ factors through $\ell_t \colon  \m{D}_t \ra \tilde{W_t}$, where $W_t:=L_t\pa{\m{D}_t}$. Up to a restriction of $\m{V}$, we can suppose that the degree of $\ell_t$, the number of branch points of $\ell_t$ and the geometric genus of $W_t$ (denoted, respectively, by $d$, $r$ and $q$) do not depend on $t\in \m{V}$. We recall that $g\pa{\m{D}_t}=2g-2$, $g\pa{\m{C}_t}=g$ and $\hat I$ is surjective on each fibre. By Riemann-Hurwitz formula, either $L_t$ is birational on the image or $g-2\leq q\leq g-1$.

Assume $q=g-2$.
%$r=2\pa{2g-3}-d\pa{2g-6}$.
By a count of moduli, either the family $\m{D}/\m{V}$ is trivial, or $d=2$ and $r=6$. In the second case $\hat{\m{B}_t}=\ker \hat I$, and consequently $J\pa{\m{C}_t}$, is isogenous to the Prym variety of the ramified double covering $\ell_t \colon  \m{D}_t \ra \tilde{W_t}$. The dimension of the moduli space of the double coverings of a curve of genus $g-2$ with $6$ branch points is $3g-3$. It follows that, in this case, the dimension of the Prym locus is equal to the dimension of the Jacobian locus. By \cite[Theorem 1.2]{PrymVP}, a very generic Prym variety of dimension at least $4$ is not isogenous to a Jacobian. This yields a contradiction.

If $q=g-1$, then $\ker\pa{J\pa{\m{D}_t}\ra J\pa{W_t}}$ contains an abelian subvariety $S_t$ of $J\pa{\m{D}_t}$ of dimension $g-1$. Since $S_t$ is contained in $\hat G(\hat {\m{B}}_t)$, then $\hat {\m{B}}_t$, and consequently $J\pa{\m{C}_t}$, is not simple. Thus we get a contradiction.

%If $q=g-1$, then $r=2\pa{2g-3}-d\pa{2g-4}$. The dimension of the space of coverings of degree $d$ of a curve of genus $g-1$ with $2\pa{2g-3}-d\pa{2g-4}$ branch points is strictly smaller then $\dim V=3g-3$. This implies that $\m{D}$ is a trivial family on a dense open set of $\m{V}$ and we get a contradiction.

\end{proof}

We want to show that the conjecture is true when $g=4,5$. To this end, we need the following result on the number of parameters of curves lying on an abelian variety. When the dimension of the abelian variety is greater than $2$ this is an improvement of the estimate in \cite[Proposition 2.4]{Parameters}.

\begin{prop}
\label{prop:def}
Let $X$ be an irreducible subvariety of $M_g$ whose points parameterize curves of geometric genus $g\geq 3$ on an abelian variety $A$ of dimension $n\geq 2$. Then $\dim X\leq g-2$ if $A$ is a surface and $\dim X \leq g-3$ if $n\geq 3$.
\end{prop}
\begin{proof}
Let $C$ be an irreducible smooth curve of genus $g$ and $\varphi\colon C \ra A$ be a morphism birational on the image. The space of the first-order deformations of $\varphi$ is $H^0\pa{C, N}$, where $N$ is the sheaf defined by the exact sequence
\begin{equation}
\label{eq:pullbackdef}
0 \ra T_C \ra \varphi^*T_A \ra N \ra 0.
\end{equation}
The sheaf $N$ is usually not locally free but there is an exact sequence
\begin{equation}
\label{eq:sky}
0 \ra S \ra N \ra N' \ra 0,
\end{equation}
where $S$ is the skyscraper sheaf, with support in the points of $C$ in which $d\varphi$ vanishes, and $N'$ is a locally free sheaf. Arbarello and Cornalba (see \cite{Petri}) proved that the infinitesimal deformations of $\varphi$ corresponding to sections of $S$ induce trivial deformations of the curve $\varphi\pa{C}$. It follows that the dimension of an irreducible sub-variety of the Hilbert scheme of curves on $A$, whose points parameterize curves of geometric genus $g$, has dimension less or equal than $h^0\pa{N'}$. This implies that $X$ has dimension at most $h^0\pa{N'}-n$.

From \eqref{eq:pullbackdef} and \eqref{eq:sky} we get
\[
c_1\pa{N'}=c_1\pa{N}-c_1\pa{S}=c_1\pa{\varphi^*T_A }- c_1\pa{T_C}-c_1\pa{S}=\omega_C\pa{-D},
\]
where $D$ is the divisor associated to the support of $S$.

If $A$ is an abelian surface, $N'$ is a line bundle and it follows that $h^0\pa{N'}\leq g$. Otherwise, $N'$ is a locally free sheaf of rank $n-1$, generated by the global sections (see \eqref{eq:pullbackdef}). Up to a suitable choice of $n-2$ independent global sections of $N'$, we have the following exact sequence
\begin{equation}
\label{eq:texidor}
  0 \ra \fascio{C}^{n-2}\ra N'\ra \omega_C\pa{-D}\ra 0.
\end{equation}
We want to prove that $\varphi\colon C \ra A$ depends on, at most, $g-3+n$ parameters. If not, $h^0\pa{N'}> g-3+n$ and the following inequality holds
\[
g-3+n<h^0\pa{N'}\leq n-2+h^0\pa{\omega_C\pa{-D}}.
\]
Thus $\deg D=0$, $N'=N$ and the exact sequence \eqref{eq:texidor} becomes
\begin{equation*}
  0 \ra \fascio{C}^{n-2}\ra N\ra \omega_C\ra 0.
\end{equation*}
By the rigidity of hyperelliptic curves in abelian varieties (see \cite[Section 2]{KummerPirola}), we can assume $C$ to be not hyperelliptic; it follows that the previous exact sequence splits and $N=\fascio{C}^{n-2}\oplus \omega_C$. The exact sequence \eqref{eq:pullbackdef} becomes
\begin{equation*}
0 \ra T_C \ra \fascio{C}^n \ra \fascio{C}^{n-2}\oplus \omega_C \ra 0.
\end{equation*}
and $T_C\subset \fascio{C}^{2}$. Thus we get a contradiction. 
% or $\varphi\colon C \ra A$  not generic, or the image is locally contained in a surface, less parameters
%$0\ra H^0\pa{\fascio{C}^n}\ra H^0\pa{\fascio{C}^{n-2}}\oplus H^0\pa{\omega_C } \ra H^1\pa{T_C}$
%le deformazioni di $C$ dipendono da $g-2$ parametri! se $C$ è ferma, visto che la mappa è birazionale anche $\varphi\pa{C}$ è ferma.
\end{proof}

\begin{theorem}
\label{teo:conj45}
There are no curves of genus $2g-2$ on a very generic Jacobian of dimension $g=4,5$.
\end{theorem}
\begin{proof}
By Proposition \ref{prop:bir}, it is sufficient to show that, for $g=4$ (resp. $g=5$) the map $L_t\colon \m{D}_t \ra \hat {\m{A}}_t$ (see \eqref{eq:L}) is not birational on the image. Set $n:=\dim \hat {\m{A}}_t=2$ (resp. $n=3$) and $w:=\pa{3g-3}-n\pa{n+1}/2=6$ (resp. $w=6$). If $L_t\colon \m{D}_t \ra \hat {\m{A}}_t$ would be birational on the image for all $t\in \m{V}$, there would be an abelian variety $A$, of dimension $n$, and an irreducible subvariety $X\subset M_{2g-2}$ of dimension $w$, whose points parameterize curves of geometric genus $2g-2=6$ (resp. $2g-2=8$) on $A$. By Proposition \ref{prop:def} we get a contradiction.
\end{proof}

\subsection{Some remarks on Conjecture \ref{conj}}

One can try to prove the conjecture with the same argument of Theorem \ref{teo:main} (see Section \ref{sec:mainTheorem}). The proof still works except for Theorem \ref{teo:xiao}, that we use in Proposition \ref{prop:norm}, whose hypotheses are no longer satisfied in the new setting. Thus, a possible approach to the conjecture is to extend the result of Theorem \ref{teo:xiao} also to the case $n\geq 4$ and $g\pa{D}\leq 2n-1$.

This is possible, for example, for $n=4$. Namely, we have the following result.

\begin{prop}
\label{prop:minixiao}
Let $J$ be a very generic Jacobian of dimension $4$, $D$ be a smooth projective curve of genus $7$ and $f\colon D \ra J$ be a non-constant map. Then the only deformations of $\pa{D,f}$, with $J$ fixed, are obtained by composing $f$ with translations.
\end{prop}

By replacing Theorem \ref{teo:xiao} by Proposition \ref{prop:minixiao}, the argument of Section \ref{sec:mainTheorem} provides an alternative proof (with respect to Theorem \ref{teo:conj45}) of the conjecture in the case $g=5$.

\begin{proof}[Proof of Proposition \ref{prop:minixiao}]
The statement is a direct result of a work by Bardelli, Ciliberto and Verra (see \cite{bardcilverra}). Consider a diagram of type
\[
\xymatrix{
\m{C}\ar[r]^{f} \ar[d]_{p}& \m{A}\ar[d]^{\pi}\\
U\ar[r]^{r} & B\\
}
\]
such that
\begin{itemize}
\item $p\colon \m{C} \ra U$ is a smooth projective family of irreducible curves of genus $7$ over a smooth connected base $U$;
\item $\pi\colon \m{A}\ra B$ is a smooth family of principally polarized abelian varieties of dimension $4$ over a smooth connected base $B$;
\item $\dim f\pa{\m{C}}=\dim r\pa{U}+1$;
\item $\dim B=10$.
\end{itemize}
In \cite[Section 3]{bardcilverra} it is proved that the natural map $B\ra \mathcal{A}_g$, to the moduli space of principally polarized abelian variety, is dense and, for each $t\in U$, $\m{A}_{r\pa{t}}$ is the Prym variety of a covering of a curve of genus $3$ by $\m{C}_t$. Since the Jacobian locus is a divisor in $\mathcal{A}_g$, it follows that a curve of genus $7$ on a very generic Jacobian of dimension $4$ is rigid.
\end{proof}

\begin{ack}
I would like to thank my Ph.D. advisor Gian Pietro Pirola for having introduced me to this subject and for all his help and patience during the preparation of the paper.
\end{ack}

%\bibliographystyle{alpha}
%
%\bibliography{biblio}

\end{document}